\documentclass[12pt]{article}

\usepackage{amssymb}
\usepackage{url}
\usepackage{amscd}
\usepackage{amsthm}
\usepackage{amsmath}
\usepackage{graphicx,tikz,float}
\usepackage[figuresleft]{rotating}

\usepackage{algorithm}
\usepackage{algorithmic}

\newtheorem {Lemma} {Lemma}
\newtheorem {Theorem}  {Theorem}

\newtheorem {Conjecture} {Conjecture}

\newtheorem {Problem} {Problem}

\newtheorem {Definition} {Definition}
\newtheorem {Fact} {Fact}

\begin{document}
\baselineskip = 15pt
\bibliographystyle{plain}

\title{\textsf{NP}-completeness of Tiling Finite Simply Connected Regions\\with a Fixed Set of Wang Tiles}
\date{}
\author{Chao Yang\\ 
              School of Mathematics and Statistics\\
              Guangdong University of Foreign Studies, Guangzhou, 510006, China\\
              sokoban2007@163.com, yangchao@gdufs.edu.cn\\
              \\
        Zhujun Zhang\\
              Big Data Center of Fengxian District, Shanghai, 201499, China\\
              zhangzhujun1988@163.com\\
              }

\maketitle

\begin{abstract}
The computational complexity of tiling finite simply connected regions with a fixed set of tiles is studied in this paper. We show that the problem of tiling simply connected regions with a fixed set of $23$ Wang tiles is \textsf{NP}-complete. As a consequence, the problem of tiling simply connected regions with a fixed set of $111$ rectangles is \textsf{NP}-complete. Our results improve that of Igor Pak and Jed Yang by using fewer numbers of tiles. Notably in the case of Wang tiles, the number has decreased by more than one third from $35$ to $23$.
\end{abstract}


\section{Introduction} 
A Wang tile is a unit square with each edge assigned a color. Wang first study the problem of tiling the entire plane with translated copies from a set of Wang tiles such that adjacent tiles must be matched by colors \cite{wang61}, which is known as Wang's domino problem. It was shown by Berger that Wang's domino problem is undecidable \cite{b66}.

When tiling finite regions with Wang tiles, the problem becomes obvious in \textsf{NP}. In fact, the tiling problems for finite regions are known to be \textsf{NP}-complete, under several different conditions. A major type of tiling problems require each tile to be used exactly once, which resembles the jigsaw puzzles. The \textsf{NP}-completeness results of this type are obtained in \cite{dd07, tw06}. In this situation, the number of tiles increases as the region becomes larger.

Another type of tiling problem allows unlimited usage of translated copies from a set of tiles, just like the original Wang's domino problem. In this case, the tiling problems are well-defined with a fixed set of tiles. The tiling problems for finite regions remain \textsf{NP}-complete even for some fixed sets of tiles. If the regions are allowed to have holes, the tiling problems for finite regions are known to be \textsf{NP}-complete for several relatively small sets of tiles \cite{hinsu17,mr01}. In particular, it is \textsf{NP}-complete for just a set of two bars \cite{bnrr95}. A \textit{bar} is an $n\times 1$ or $1 \times m$ polyomino which can be viewed as gluing several Wang tiles together by unique internal colors (and all external sides receive the same color).

\begin{Theorem}[Tiling with two bars \cite{bnrr95}]\label{thm_2bar}
   Let $n,m\geq 2$ be two fixed integers ($n,m$ are not both $2$), and let $h_n$ and $v_m$ be a horizontal bar of length $n$, and a vertical bar of height $m$, respectively. Then the tiling problem for general finite regions (allowing holes) with the two bars $h_n$ and $v_m$ is \textsf{NP}-complete. 
\end{Theorem}

Theorem \ref{thm_2bar} implies that tiling general finite regions with a fixed set of $5$ Wang tiles (by breaking down two bars $h_3$ and $v_2$) is \textsf{NP}-complete. On the other hand, for tiling simply connected regions (i.e. without holes), more tiles are needed to obtain \textsf{NP}-completeness result. Jed Yang proved the following theorem \cite{y_thesis}.

\begin{Theorem}[\cite{y_thesis}]
    There exists a set $W$ of $35$ Wang tiles such that the tiling problem for finite simply connected regions with $W$ is \textsf{NP}-complete. 
\end{Theorem}

The main contribution of this paper is to decrease the total number of Wang tiles yet preserving the \textsf{NP}-completeness by proving Theorem \ref{thm_main}.

\begin{Theorem}\label{thm_main}
    There exists a set $W$ of $23$ Wang tiles such that the tiling problem for finite simply connected regions with $W$ is \textsf{NP}-complete. 
\end{Theorem}


Below is a formal definition of the main problem being studied in this paper. Because we only study the tiling problems for finite (i.e. bounded) regions in this paper, the word \textit{finite} is sometimes omitted in the rest of the paper.

\begin{Definition}[Tiling finite simply connected regions with a fixed set of tiles]
Let $T$ be a fixed set of tiles. Given an arbitrary finite simply connected region $D$, is it possible to tile the region $D$ with translated copies from $T$?
\end{Definition}

We focus on finding sets $T$ with as fewer number of tiles as possible yet the tiling problems for simply connected regions with $T$ are \textsf{NP}-complete. This number depends on different kinds of tiles. Besides Wang tiles, there are several other kinds of tiles which have received considerable studies. A \textit{polyomino} is a simply connected tile by gluing several unit squares together edge-to-edge. A \textit{generalized Wang tile} is a polyomino with a color assigned to each unit segment on its boundary. A \textit{rectangle} is a rectangular polyomino (this excludes rectangles with irrational ratios). When restricted to these four kinds of tiles (Wang tiles, generalized Wang tiles, polyominoes, and rectangles), the simply connected region $D$ we considered can be also restricted to polyomino regions. For Wang tiles or generalized Wang tiles, the regions always implicitly include colors assigned to each of the unit segments of the boundaries. Therefore, in addition to the shape of the region, the tilings must also match the colors on the boundary of the region.

The second column of Table \ref{tbl_min} summarizes the known minimum sizes (i.e. the number of tiles in $T$) of tile sets $T$  that tiling simply connected regions with $T$ is \textsf{NP}-complete, for different kinds of tiles \cite{py13, py13arxiv,y_thesis}.

\begin{table}[H]
\begin{center}
\begin{tabular}{|l|r|r|}
\hline
Kinds of tiles          & Known minimum sizes  & Sizes of our results \\ \hline
Wang tiles             & 35      & 23                        \\ \hline
Generalized Wang tiles & 15         & 14    and 8                \\ \hline
Rectangles             & 117            & 111                 \\ \hline
\end{tabular}
\end{center}
\caption{Sizes of $T$  that tiling simply connected regions with $T$ is \textsf{NP}-complete.}\label{tbl_min}
\end{table}

In this paper, all the known minimum sizes in the second column of Table \ref{tbl_min} have been improved as shown in the third column. The main result, Theorem \ref{thm_main}, is proved following \cite{py13, py13arxiv,y_thesis} by reduction from the \textsf{NP}-complete problem \textsc{Cubic Monotone 1-in-3 SAT} \cite{mr01}. We introduce several novel techniques in the reduction to decrease the number of Wang tiles.

\begin{Definition}[\textsc{Cubic Monotone 1-in-3 SAT}]
An instance of \textsc{Cubic Monotone 1-in-3 SAT} problem is a \textsc{3SAT} instance satisfying the following additional conditions: (1) each variable appears in the formula exactly $3$ times (i.e. cubic); (2) none of the variables appears in negation (i.e. monotone, also referred to as positive). (3) Furthermore, an instance is satisfied if and only if exactly $1$ of every $3$ variables in every clause is true (i.e. $1$-in-$3$). 
\end{Definition}

By the definition of \textsc{Cubic Monotone 1-in-3 SAT}, every instance must have the same number of variables and clauses. An example instance is $\varphi=(x_1 \vee x_1 \vee x_3) \wedge (x_2 \vee x_2 \vee x_3) \wedge (x_1 \vee x_2 \vee x_3)$.

The rest of the paper is organized as follows. Section \ref{sec_wang} gives the proof of Theorem \ref{thm_main}. Section \ref{sec_rect} proves a variant of Theorem \ref{thm_main}. With this variant, we show that tiling simply connected regions with a set of $111$ rectangles is \textsf{NP}-complete. Section \ref{sec_con} concludes with a few remarks.

\section{Tiling with $23$ Wang Tiles} \label{sec_wang} 


We will prove Theorem \ref{thm_main} in this section in two steps. First, we construct a set $W$ of $23$ Wang tiles, as illustrated in Figure \ref{fig_wang_23}. Secondly, the \textsf{NP}-hardness of tiling simply connected region with $W$ is shown by reduction from the \textsc{Cubic Monotone 1-in-3 SAT} problem. For each instance $\varphi$ of \textsc{Cubic Monotone 1-in-3 SAT} problem, we will construct a region $D$ such that formula $\varphi$ is (1-in-3) satisfied if and only if the region $D$ can be tiled by the set $W$ of Wang tiles. The novel techniques that we incorporate to the general reduction method of Pak and Yang \cite{py13,py13arxiv,y_thesis} will be summarized after the proof of Theorem \ref{thm_main}.


\begin{figure}[H]
\begin{center}
\begin{tikzpicture}


\draw (0,0)--(1,0)--(1,3)--(0,3)--(0,0);

\foreach \y in {1,2}
{
\draw (1,\y)--(0,\y);
\node at (0.2,\y+0.5) {$v$}; 
\node at (0.8,\y+0.5) {$0$};
}

\draw (0,0)--(0.5,0.5)--(1,0);\draw (0,3)--(0.5,2.5)--(1,3);
\draw (0.5, 0.5)--(0.5,2.5);
\node at (0.5,0.2) {$b$}; \node at (0.5,2.8) {$b$};
\node at (0.2,0.5) {$v$}; 
\node at (0.8,0.5) {$0$};


\draw (4,0)--(3,0)--(3,1)--(4,1)--(4,0)--(3,1);
\draw (3,0)--(4,1);
\node at (3.5,0.2) {$b$}; 
\node at (3.5,0.8) {$b$}; 
\node at (3.8,0.5) {$1$}; 
\node at (3.2,0.5) {$v$}; 


\foreach \x in {9}
{
\draw (\x+0,0)--(\x+1,0)--(\x+1,2)--(\x+0,2)--(\x+0,0);
\draw (\x+0,1)--(\x+1,1);

\draw (\x+0,0)--(\x+0.5,0.5)--(\x+1,0);\draw (\x+0,2)--(\x+0.5,1.5)--(\x+1,2);
\draw (\x+0.5, 0.5)--(\x+0.5,1.5);
\node at (\x+0.5,0.2) {$b$}; \node at (\x+0.5,1.8) {$b$};
\node at (\x+0.2,0.5) {$0$}; 
\node at (\x+0.8,0.5) {$1$};
\node at (\x+0.2,1.5) {$1$}; 
\node at (\x+0.8,1.5) {$0'$};
}

\foreach \x in {3}
{
\draw (\x+4,0)--(\x+3,0)--(\x+3,1)--(\x+4,1)--(\x+4,0)--(\x+3,1);
\draw (\x+3,0)--(\x+4,1);
\node at (\x+3.5,0.2) {$b$}; 
\node at (\x+3.5,0.8) {$b$}; 
\node at (\x+3.8,0.5) {$0'$}; 
\node at (\x+3.2,0.5) {$0$}; 

}


\foreach \y in {-3}
\foreach \x in {0}
{
\draw (\x+4,0+\y)--(\x+3,0+\y)--(\x+3,1+\y)--(\x+4,1+\y)--(\x+4,0+\y)--(\x+3,1+\y);
\draw (\x+3,0+\y)--(\x+4,1+\y);
\node at (\x+3.5,0.2+\y) {$l$}; 
\node at (\x+3.5,0.8+\y) {$l$}; 
\node at (\x+3.8,0.5+\y) {$i$}; 
\node at (\x+3.2,0.5+\y) {$i$}; 
}

\foreach \y in {-3}
\foreach \x in {5.5}
{
\draw (\x+0,0+\y)--(\x+2,0+\y)--(\x+2,1+\y)--(\x+0,1+\y)--(\x+0,0+\y);
\draw (\x+1,1+\y)--(\x+1,0+\y);

\draw (\x+0,0+\y)--(\x+0.5,0.5+\y)--(\x+0,1+\y);\draw (\x+2,0+\y)--(\x+1.5,0.5+\y)--(\x+2,1+\y);
\draw (\x+0.5, 0.5+\y)--(\x+1.5,0.5+\y);
\node at (\x+0.2,0.5+\y) {$i$}; 
\node at (\x+1.8,0.5+\y) {$i$};
\node at (\x+0.5,0.8+\y) {$b$}; 
\node at (\x+0.5,0.2+\y) {$r$};
\node at (\x+1.5,0.8+\y) {$r$}; 
\node at (\x+1.5,0.2+\y) {$b$};
}


\foreach \x in {-3}
\foreach \y in {-3}
{
\draw (\x+4,0+\y)--(\x+3,0+\y)--(\x+3,1+\y)--(\x+4,1+\y)--(\x+4,0+\y)--(\x+3,1+\y);
\draw (\x+3,0+\y)--(\x+4,1+\y);
\node at (\x+3.5,0.2+\y) {$b$}; 
\node at (\x+3.5,0.8+\y) {$b$}; 
\node at (\x+3.8,0.5+\y) {$i$}; 
\node at (\x+3.2,0.5+\y) {$i$}; 
}


\foreach \x in {9}
\foreach \y in {-3}
{
\draw (\x+0,0+\y)--(\x+1,0+\y)--(\x+1,2+\y)--(\x+0,2+\y)--(\x+0,0+\y);
\draw (\x+0,1+\y)--(\x+1,1+\y);

\draw (\x+0,0+\y)--(\x+0.5,0.5+\y)--(\x+1,0+\y);\draw (\x+0,2+\y)--(\x+0.5,1.5+\y)--(\x+1,2+\y);
\draw (\x+0.5, 0.5+\y)--(\x+0.5,1.5+\y);
\node at (\x+0.5,0.2+\y) {$l$}; \node at (\x+0.5,1.8+\y) {$r$};
\node at (\x+0.2,0.5+\y) {$i$}; 
\node at (\x+0.8,0.5+\y) {$j$};
\node at (\x+0.2,1.5+\y) {$j$}; 
\node at (\x+0.8,1.5+\y) {$i$};
}

\node at (1.5,0.2) {$V_0$}; 
\node at (4.5,0.2) {$V_1$}; 
\node at (7.5,0.2) {$C_0$}; 
\node at (10.5,0.2) {$C_1$};

\node at (4.5,-2.8) {$L_i$}; 
\node at (8,-2.8) {$R_i$}; 
\node at (1.5,-2.8) {$F_i$}; 
\node at (10.5,-2.8) {$X_{ij}$};

\end{tikzpicture}
\end{center}
\caption{The variables ($V_0, V_1$), clauses ($C_0, C_1$), forwarders ($F_i, i=0,1$), left anchors ($L_i, i=0,1$) and  right anchors ($R_i, i=0,1$), and crossovers ($X_{ij}, i,j=0,1$).}\label{fig_wang_23}
\end{figure}

\begin{proof}[Proof of Theorem \ref{thm_main}]
The $23$ Wang tiles (see Figure \ref{fig_wang_23}) are glued together by unique colors so that they must be used in groups as generalized Wang tiles (some generalized Wang tiles consist of a single Wang tile). There are five types of generalized Wang tiles. We will introduce each of them and their intended usage in simulating the formulas. Since the tiles need to match the boundary of the region, their functionality can be fully understood only after the region has been constructed according to a given formula $\varphi$.

The first type is the \textit{variable} tiles which are responsible for assigning truth value to each of the variables of the formula $\varphi$. The variable tile $V_0$ is a vertical bar of height $3$, and the variable tile $V_1$ is a single Wang tile. The west sides of both variable tiles are assigned color $v$, which only appears in the variable tiles. This unique color will make sure that they can only be placed next to the left boundary of the region which will be introduced later. Intuitively speaking, the variable tiles set the truth values which will be sent to the right across the region to be tiled. The colors $1$ and $0$ represent true and false, respectively.

The second type is the \textit{clause} tiles $C_0$ and $C_1$. There is also a color that only appears in the clause tiles, namely $0'$ (to be clear, $0$ and $0'$ are two different colors). This color guarantees that the clause tiles must be placed next to the right boundary of the region, where the clause tiles will check each clause of $\varphi$ by receiving the truth values of the variables coming from the left.

The third type is the \textit{forwarders} $F_i (i=1,2)$, which also has two variants. The forwarders just let the signals, either $1$ or $0$, pass from left to right.

The fourth type is the anchor tiles. The \textit{left anchors} $L_i (i=1,2)$ and the \textit{right anchors} $R_i (i=1,2)$ work together to control the location at which two parallel signals (truth values) should be swapped. The right anchors and the left anchors are placed starting at the top and bottom boundary of the region, respectively, and extending to the interior. The swap of two signals takes place at the location where the left and right anchors meet at the interior of the region. Both the left anchors and the right anchors come in two variants, differing by the color on the vertical sides (east and west sides). The colors on the vertical sides are either both $1$ or both $0$. This allows the signals to pass the anchors from left to right.

The fifth and the last type is the \textit{crossovers} $X_{ij} (i,j=1,2)$. The crossovers must be put at the places where the left anchors and the right anchors meet. The job of the crossovers is to swap the signals on two adjacent rows. This is done by assigning proper colors on the west and east sides of each crossover tile. For example, the colors on the west and east sides of the crossover tile $X_{10}$ are $01$ and $10$, respectively. Therefore it swaps two adjacent signals of $0$ and $1$. There are four variants of crossovers, corresponding to four different combinations of two adjacent signals.

We have finished introducing our fixed set of generalized Wang tiles. Breaking down this set of generalized Wang tiles, we get a set $W$ of $23$ Wang tiles.

Now, for each instance $\varphi$ of \textsc{Cubic Monotone 1-in-3 SAT} problem, we construct a finite simply connected region for the tiling problem with the above fixed set $W$ of $23$ Wang tiles. Without loss of generality, we take the formula 
$$\varphi=(x_1 \vee x_1 \vee x_3) \wedge (x_2 \vee x_2 \vee x_3) \wedge (x_1 \vee x_2 \vee x_3)$$
as an example to explain the construction of the region in detail. The region constructed for this formula $\varphi$ is shown in Figure \ref{fig_region}. The construction method could be extended to the general cases naturally. 

The region consists of three parts (see Figure \ref{fig_region}): the left part for variable tiles (in light orange), the right part for clause tiles (in light violet), and the central part (in green) for the signals to be transmitted from the variable part to the clause part. The central part can be further divided into several rectangular sub-regions (as illustrated in darker or lighter green). Each sub-region will take care of the crossover of a pair of adjacent signals.

Since every variable appears in $\varphi$ exactly three times, we have three signals for each variable. As we have already seen in Figure \ref{fig_region}, the orange area on the left is separated into $3$ sub-regions by a redundant row between every two adjacent sub-regions. The $3$ sub-regions represent the $3$ variables $x_1$, $x_2$ and $x_3$, respectively, from top to bottom. Each sub-region has exactly three squares, so it can be occupied by either a single variable tile $V_0$, or three copies of the variable tile $V_1$. In other words, the $3$ signals for the same variable must be set to either all $1$ or all $0$. The two redundant signals separating the $3$ sub-regions are denoted by $z_1$ and $z_2$, and their truth value does not matter. We can simply assign $0$ to all the redundant variables by setting the color $0$ to the unit segments of the left and right boundary of the region which correspond to the redundant variables.

\begin{figure}[H]
\begin{center}
\begin{tikzpicture}[scale=0.6]

\draw [fill=orange!20,draw=none] (-1,0)--(0,0)--(0,3)--(-1,3)--(-1,0);
\draw [fill=orange!20,draw=none] (-1,4)--(0,4)--(0,7)--(-1,7)--(-1,4);
\draw [fill=orange!20,draw=none] (-1,8)--(0,8)--(0,11)--(-1,11)--(-1,8);

\draw [fill=lime!20,draw=none]  (0,0)--(8,0)--(8,11)--(0,11)--(0,0);
\draw [fill=lime!40,draw=none]  (8,0)--(12,0)--(11,5)--(12,9)--(11,11)--(8,11)--(8,0);
\draw [fill=lime!40,draw=none]  (12.5,0)--(11.5,5)--(12.5,9)--(11.5,11)--(15,11)--(15,0)--(12.5,0);
\draw [fill=lime!20,draw=none]  (15,0)--(15,11)--(20,11)--(20,0)--(15,0);

\foreach \y in {0,4,8}
{
\draw [fill=violet!20,draw=none] (20,0+\y)--(22,0+\y)--(22,2+\y)--(21,2+\y)--(21,3+\y)--(20,3+\y)--(20,0+\y);
}

\draw (-1,11)--(-1,8)--(0,8)--(0,7)--(-1,7)--(-1,4)--(0,4)--(0,3)--(-1,3)--(-1,0)--(12,0)--(11,5)--(12,9)--(11,11)--(-1,11);
\draw (12.5,0)--(11.5,5)--(12.5,9)--(11.5,11)--(21,11)--(21,10)--(22,10)--(22,8)--(20,8)--(20,7)--(21,7)--(21,6)--(22,6)--(22,4)--(20,4)--(20,3)--(21,3)--(21,2)--(22,2)--(22,0)--(12.5,0);

\foreach \y in {1,2,3,5,6,7,9,10,11}
{
\node at (-1.3,\y-0.5) {$v$}; 
}

\foreach \x in {0,21}
\foreach \y in {-0.2, 3.2, 3.8, 7.2, 7.8, 11.2}
{
\node at (\x-0.5,\y) {\small $b$}; 
}

\foreach \x in {22}
\foreach \y in {-0.2, 2.2, 3.8, 6.2, 7.8, 10.2}
{
\node at (\x-0.5,\y) {\small $b$}; 
}

\foreach \y in {3.5, 7.5}
{
\node at (-0.2,\y) {\small $0$}; 
}
\foreach \y in {3.5, 7.5}
{
\node at (20.2,\y) {\small $0$}; 
}

\foreach \y in {3,7,11}
{
\node at (21.3,\y-0.5) {\small $0'$}; 
\node at (22.3,\y-1.5) {\small $0'$}; 
\node at (22.3,\y-2.5) {\small $1$}; 
}

\node at (-3,10.5) {$x_1$};
\node at (-3,9.5) {$x_1$};
\node at (-3,8.5) {$x_1$};
\node at (-3,7.5) {$z_1$};
\node at (-3,6.5) {$x_2$};
\node at (-3,5.5) {$x_2$};
\node at (-3,4.5) {$x_2$};
\node at (-3,3.5) {$z_2$};
\node at (-3,2.5) {$x_3$};
\node at (-3,1.5) {$x_3$};
\node at (-3,0.5) {$x_3$};

\node at (24,10.5) {$x_1$};
\node at (24,9.5) {$x_1$};
\node at (24,8.5) {$x_3$};
\node at (24,7.5) {$z_1$};
\node at (24,6.5) {$x_2$};
\node at (24,5.5) {$x_2$};
\node at (24,4.5) {$x_3$};
\node at (24,3.5) {$z_2$};
\node at (24,2.5) {$x_1$};
\node at (24,1.5) {$x_2$};
\node at (24,0.5) {$x_3$};

\end{tikzpicture}
\end{center}
\caption{The simply connected region corresponding to $\varphi$.}\label{fig_region}
\end{figure}

In all, $11$ signals are transmitting from the left boundary to the right in parallel, and they are $x_1$, $x_1$, $x_1$, $z_1$, $x_2$, $x_2$, $x_2$, $z_2$, $x_3$, $x_3$ and $x_3$, from top to bottom.

\begin{figure}[H]
\begin{center}
\begin{tikzpicture}[scale=0.6]

\draw (0,0)--(2,0)--(2,2)--(1,2)--(1,3)--(0,3)--(0,0); \draw (0,2)--(1,2)--(1,0); \draw (0,1)--(2,1);
\node at (1.3,2.5) {\small $0'$}; 
\node at (2.3,1.5) {\small $0'$}; 
\node at (2.3, 0.5) {\small $1$}; 

\node at (0.5, 2.5) {\small $C_0$}; \node at (0.5, 1.5) {\small $F_0$};  \node at (1.5, 1.5) {\small $C_0$}; 
\node at (0.5, 0.5) {\small $F_1$}; \node at (1.5, 0.5) {\small $F_1$}; 

\foreach \x in {5}
{
\draw (\x+0,0)--(\x+2,0)--(\x+2,2)--(\x+1,2)--(\x+1,3)--(\x+0,3)--(\x+0,0); \draw (\x+0,2)--(\x+1,2)--(\x+1,0); \draw (\x+0,1)--(\x+1,1);
\node at (\x+1.3,2.5) {\small $0'$}; 
\node at (\x+2.3,1.5) {\small $0'$}; 
\node at (\x+2.3, 0.5) {\small $1$}; 

\node at (\x+0.5, 2.5) {\small $C_0$}; \node at (\x+0.5, 1.5) {\small $F_1$};  \node at (\x+1.5, 1) {\small $C_1$}; 
\node at (\x+0.5, 0.5) {\small $F_0$};  

}

\foreach \x in {10}
{
\draw (\x+0,0)--(\x+2,0)--(\x+2,2)--(\x+1,2)--(\x+1,3)--(\x+0,3)--(\x+0,0); \draw (\x+1,2)--(\x+1,2)--(\x+1,0); \draw (\x+0,1)--(\x+1,1);
\node at (\x+1.3,2.5) {\small $0'$}; 
\node at (\x+2.3,1.5) {\small $0'$}; 
\node at (\x+2.3, 0.5) {\small $1$}; 

\node at (\x+0.5, 2) {\small $C_1$};    \node at (\x+1.5, 1) {\small $C_1$}; 
\node at (\x+0.5, 0.5) {\small $F_0$};  

}


\foreach \y in {0,1,2}
{
\draw [->, thick] (-1,0.5+\y)--(0,0.5+\y);
}
\node at (-1.2, 0.5) {\small $1$};  
\node at (-1.2, 1.5) {\small $0$};  
\node at (-1.2, 2.5) {\small $0$};

\foreach \x in {5}
{
\foreach \y in {0,1,2}
{
\draw [->, thick] (\x+-1,0.5+\y)--(\x+0,0.5+\y);
}
\node at (\x+-1.2, 0.5) {\small $0$};  
\node at (\x+-1.2, 1.5) {\small $1$};  
\node at (\x+-1.2, 2.5) {\small $0$};  
}

\foreach \x in {10}
{
\foreach \y in {0,1,2}
{
\draw [->, thick] (\x+-1,0.5+\y)--(\x+0,0.5+\y);
}
\node at (\x+-1.2, 0.5) {\small $0$};  
\node at (\x+-1.2, 1.5) {\small $0$};  
\node at (\x+-1.2, 2.5) {\small $1$};  
}

\end{tikzpicture}
\end{center}
\caption{The clause sub-regions.}\label{fig_clause_region}
\end{figure}

The violet area on the right of the region is also separated into sub-regions by redundant rows. Each sub-region corresponds to a clause of $\varphi$. For example, the first sub-region of violet area will check the first clause of $\varphi$, namely $x_1 \vee x_1 \vee x_3$. More precisely, the first, second, and third rows of the first sub-region will receive the signals $x_1$, $x_1$, and $x_3$, respectively. By the definition of \textsc{Cubic Monotone 1-in-3 SAT} problem, a formula is (1-in-3) satisfied if and only if exactly one variable of three is true for every clause. So the east sides of the boundary of each sub-region of the clause area are colored $0'$, $0'$, and $1$. If the unique signal $1$ comes in the third row of a sub-region, we can just use $F_i (i=1,2)$ and $C_0$ to transmit the signals to the right boundary (see the left of Figure \ref{fig_clause_region}). If the signal $1$ comes in the second row, use a copy of $C_1$ to swap the signal $1$ one unit downwards to the right boundary (see the middle of Figure \ref{fig_clause_region}). If the signal $1$ come in the first row, use two copies of $C_1$ to swap the signal $1$ two units downwards to the right boundary (see the right of Figure \ref{fig_clause_region}). Note also that $C_0$ and $C_1$ have a color of $0'$ so that they can be placed only within the clause area. Therefore, they can be used to redirect a true signal (i.e. a signal $1$) only within each clause.

On the other hand, if the number of true signals coming into a clause is $0$, $2$, or $3$, then there is no way we can fill the clause area with the $23$ Wang tiles. So we have the following fact.

\begin{Fact}\label{fct_valid}
    A sub-region for a clause can be tiled with the set $W$ of $23$ Wang tiles if and only if there is exactly one true signal coming from the three rows on its left.
\end{Fact}

Because each variable appears exactly three times in the formula $\varphi$, there is a natural bijection between the signals emitted from the left and the signals to be received to the right. For the redundant variables we have introduced to separate variables and clauses, we just map them to the same row on the right. So from top to bottom, the signals to be received for the right part are $x_1$, $x_1$, $x_3$, $z_1$, $x_2$, $x_2$, $x_3$, $z_2$, $x_1$, $x_2$ and $x_3$, which is a permutation of the signals emitted from the variable area.

It is a well-known fact that every permutation is the composition of transpositions of two adjacent elements. Moreover, a permutation of order $n$ is the composition of no more than $n^2$ adjacent transpositions (this implies the reduction can be done within polynomial time). Let $\textsf{swap}(k)$ be the adjacent transposition of the $k$-th and $(k+1)$-th positions in our sequence of signals (i.e. signals on the $k$-th and $(k+1)$-th rows). Then by applying transpositions $\textsf{swap}(8)$, $\textsf{swap}(7)$, $\textsf{swap}(6)$, $\textsf{swap}(5)$, $\textsf{swap}(4)$, $\textsf{swap}(3)$, $\textsf{swap}(4)$, $\textsf{swap}(5)$, $\textsf{swap}(6)$, $\textsf{swap}(9)$, $\textsf{swap}(8)$, $\textsf{swap}(7)$, $\textsf{swap}(9)$ and $\textsf{swap}(8)$ to signal sequence  
$$x_1x_1 x_1 z_1 x_2 x_2 x_2 z_2 x_3 x_3 x_3$$ 
emitted from the left, we get the signal sequence  
$$x_1 x_1 x_3 z_1 x_2 x_2 x_3 z_2 x_1 x_2 x_3$$ 
which we would like to receive to the right.

\begin{figure}[H]
\begin{center}
\begin{tikzpicture}[scale=0.6]

\draw [fill=lime!20] (0,0)--(8,0)--(8,11)--(0,11)--(0,0);

\foreach \y in {1,2,5,6,7,8,9,10,11}
{
\draw [->, thick] (-0.5,\y-0.5)--(8.5, \y-0.5);
}

\foreach \x in {0,...,6}
{
\draw [fill=white] (-\x+6,11-\x)--(-\x+8,11-\x)--(-\x+8,10-\x)--(-\x+6,10-\x)--(-\x+6,11-\x);
\node at (7-\x, 10.5-\x) {\small $R$}; 
}

\draw [fill=white] (0,0)--(0,1)--(1,1)--(1,0)--(0,0);
\node at (0.5, 0.5) {\small $L$}; 

\draw [fill=white] (0,1)--(0,2)--(1,2)--(1,1)--(0,1);
\node at (0.5, 1.5) {\small $L$}; 

\draw [fill=white] (1,2)--(1,4)--(0,4)--(0,2)--(1,2);
\node at (0.5,3.5) {\small $X$}; 
\draw [->, thick, red] (-0.5, 3.5)--(0,3.5)--(1, 2.5)--(8.5, 2.5);
\draw [->, thick, blue] (-0.5, 2.5)--(0,2.5)--(1, 3.5)--(8.5, 3.5);

\foreach \x in {1,...,7}
{
\node at (\x-0.5,11.3) {\small $b$}; 
\node at (\x+0.5,-0.3) {\small $b$}; 
}
\node at (7.5,11.3) {\small $r$}; 
\node at (0.5,-0.3) {\small $l$}; 

\end{tikzpicture}
\end{center}
\caption{The crossover sub-region.}\label{fig_crossover_region}
\end{figure}

Swapping two adjacent signals can be achieved by a sub-region illustrated in Figure \ref{fig_crossover_region}. The width of this rectangular sub-region is $8$, and it forces the signals on $8$-th and $9$-th rows to swap their positions. This is done by the colors of the top and bottom boundaries. The top boundary is colored $b^7r$ and the bottom boundary is colored $l b^7$, from left to right. The only tile that can be placed directly below the color $r$ on the top boundary is a right anchor tile $R$ (tile $R$ refers to either $R_1$ or $R_2$, we sometimes omit the subscripts of other types of tiles similarly in the rest of the paper). Then another right anchor must be placed below the first one, with one unit shifted to the left. By repeating this process, we have an array of right anchors extending along the southwest direction. Similarly, an array of left anchors $L$ must be placed above the color $l$ on the bottom boundary and extend to the north. But these two arrays of anchors will cross each other, and there is a conflict in color when they are about to meet. The only way to resolve this conflict is to put a crossover tile $X$, which links the colors $r$ and $l$, between the two arrays when they are two units apart. In short, the locations of the colors $r$ and $l$ on the top and bottom boundaries determine the location of the crossover tile $X$. In this example illustrated by Figure \ref{fig_crossover_region}, the $X$ tile is placed on the $8$-th and $9$-th rows, so the $8$-th and $9$-th signals are swapped when they are transmitted through the crossover tile $X$.

For all other areas (illustrated in light green) in Figure \ref{fig_crossover_region}, only the forwarder tiles $F$ can be placed there. Note that all the tiles $R$, $L$, $X$ and $F$ have different variants, and the signals coming from the left determine which variant to be used. The tiling of the crossover sub-region can be described intuitively as a crossover tile $X$ being fixed by two kinds of anchors tiles ($R$s and $L$s) in an ocean of forwarders ($F$s).

In general, a sub-region is similar to Figure \ref{fig_crossover_region} but with a different width $w$ will swap the signals on the $w$-th and $(w+1)$-th rows. We put appropriate crossover sub-region consecutively between the variable area on the left and the clause area on the right to complete the construction of the simply connected region associated with the formula $\varphi$. In a sense, the formula is encoded by the boundary of the region, especially the boundary on the top and at the bottom, which encodes all the adjacent transpositions needed to guide all the signals from the left to the right.

Each possible (partial) tiling starting from the left of the region represents an assignment of all the variables of the formula $\varphi$. By Fact \ref{fct_valid}, this partial tiling can be extended to the right and fill up the whole region in the end if and only if exactly one out of three variables is true for every clause. This completes the proof. 
\end{proof}

We summarize the key and novel techniques used in the proof of Theorem \ref{thm_main} to decrease the number of Wang tiles below.

\begin{itemize}
    \item Redundant signals are added between the real variables of the formula. This makes an independent sub-region for each variable on the left of the region. As a consequence, we are able to decrease the number of Wang tiles for assigning the true values to variables.
    \item Utilizing both the top and bottom boundary to encode the location of the transposition, compared to Jed Yang's original method \cite{y_thesis} which only the top boundary is used. This enables us to decrease the number of Wang tiles for the crossovers.
    \item The right boundary of the region is also optimized. As a result, the number of Wang tiles for the clauses is decreased.
\end{itemize}

Note also that, in terms of generalized Wang tiles, Theorem \ref{thm_main} states that there is a set of $14$ generalized Wang tile such that tiling simply connected regions with this set is \textsf{NP}-complete. This improves the result in \cite{py13} where $15$ generalized Wang tiles are needed. It is also mentioned in \cite{py13} that by using Ollinger's method \cite{o09} to simulate any fixed set of Wang tiles by a set of $11$ polyominoes (which can be viewed as generalized Wang tiles), the number can be further decreased to $11$. Ollinger's result has been improved to $8$ by the authors of the current paper in \cite{yz24}, which implies that there is a set of $8$ generalized Wang tile such that tiling simply connected regions with this set is \textsf{NP}-complete.

\section{Tiling with Rectangles}\label{sec_rect}

In this section, we prove the \textsf{NP}-completeness of tiling simply connected regions with a set $R$ of $111$ rectangles, which improves the $117$ rectangles result of Jed Yang \cite{y_thesis}. To achieve this improvement, we need a variant of Theorem \ref{thm_main}, where the Wang tiles cannot have the same color on their parallel sides. In other words, for each Wang tile, the colors of the north and south sides must be different, and the colors of the west and east sides must also be different.

\begin{Lemma}\label{lem_wang}
    There exists a set $T$ of $29$ Wang tiles whose parallel sides are assigned different colors such that the tiling problem for finite simply connected regions with $T$ is \textsf{NP}-complete. 
\end{Lemma}  
    
\begin{proof}
The proof of Lemma \ref{lem_wang} is based on Theorem \ref{thm_main}. Both the set of tiles (Figure \ref{fig_wang_2}) and the simply connected region (Figure \ref{fig_region_2}) are modified slightly from that of Theorem \ref{thm_main} to accommodate the new restriction that parallel sides have to be assigned different colors for each Wang tile.

The key idea is the modification of the way that signals are being transmitted from left to right, originally developed in \cite{y_thesis}. Instead of forwarding the signal unchanged, the signals are now being transmitted by switching alternatively between $0$ and $1$ when passing through each Wang tile. This is achieved by modifying the tiles accordingly. Both the forwarders and left anchors are changed such that the two vertical sides of them are colored differently, namely one vertical side is $0$ and the other side is $1$ (see Figure \ref{fig_wang_2}). Therefore, by checking the signals of every other tile along transmission path from left to right, we can receive the original signal being emitted. To make sure that the signals have been transmitted to the right correctly, we can set the width of the central part of the region to be even, see Figure \ref{fig_region_2}. This can be done by padding a column (consisting of only tiles $F$ and $L$) to the central part to adjust the parity of its width if needed.

\begin{figure}[H]
\begin{center}
\begin{tikzpicture}


\draw (0,0)--(1,0)--(1,3)--(0,3)--(0,0);

\foreach \y in {1,2}
{
\draw (1,\y)--(0,\y);
\node at (0.2,\y+0.5) {$v$}; 
\node at (0.8,\y+0.5) {$0$};
}

\draw (0,0)--(0.5,0.5)--(1,0);\draw (0,3)--(0.5,2.5)--(1,3);
\draw (0.5, 0.5)--(0.5,2.5);
\node at (0.5,0.2) {$b$}; \node at (0.5,2.8) {$l$};
\node at (0.2,0.5) {$v$}; 
\node at (0.8,0.5) {$0$};

\foreach \x in {0}
{
\draw (\x+4,0)--(\x+3,0)--(\x+3,1)--(\x+4,1)--(\x+4,0)--(\x+3,1);
\draw (\x+3,0)--(\x+4,1);
\node at (\x+3.5,0.2) {$b$}; 
\node at (\x+3.5,0.8) {$l$}; 
\node at (\x+3.8,0.5) {$1$}; 
\node at (\x+3.2,0.5) {$v$}; 
}

\foreach \x in {0}
\foreach \y in {2}
{
\draw (\x+4,0+\y)--(\x+3,0+\y)--(\x+3,1+\y)--(\x+4,1+\y)--(\x+4,0+\y)--(\x+3,1+\y);
\draw (\x+3,0+\y)--(\x+4,1+\y);
\node at (\x+3.5,0.2+\y) {$l$}; 
\node at (\x+3.5,0.8+\y) {$b$}; 
\node at (\x+3.8,0.5+\y) {$1$}; 
\node at (\x+3.2,0.5+\y) {$v$}; 
}


\foreach \x in {3}
{
\draw (\x+4,0)--(\x+3,0)--(\x+3,1)--(\x+4,1)--(\x+4,0)--(\x+3,1);
\draw (\x+3,0)--(\x+4,1);
\node at (\x+3.8,0.5) {$0'$}; 
\node at (\x+3.2,0.5) {$0$}; 
\node at (\x+3.5,0.2) {$l$}; 
\node at (\x+3.5,0.8) {$b$}; 
}

\foreach \x in {9}
{
\draw (\x+1,0)--(\x+0,0)--(\x+0,2)--(\x+1,2)--(\x+1,0);
\draw (\x+0,1)--(\x+2,1)--(\x+2,0)--(\x+1,0);

\draw (\x+0,0)--(\x+0.5,0.5)--(\x+1.5,0.5)--(\x+2,0);\draw (\x+0,2)--(\x+0.5,1.5)--(\x+1,2);
\draw (\x+0.5, 0.5)--(\x+0.5,1.5); \draw (\x+1.5,0.5)--(\x+2,1);
\node at (\x+0.2,0.5) {$0$}; 
\node at (\x+1.8,0.5) {$1$};
\node at (\x+0.2,1.5) {$1$}; 
\node at (\x+0.8,1.5) {$0'$};

\node at (\x+0.5,0.2) {$b$};
\node at (\x+0.5,1.8) {$b$};
\node at (\x+1.5,0.8) {$l$};
\node at (\x+1.5,0.2) {$b$};
}


\foreach \x in {-5}
\foreach \y in {-3}
{
\draw (\x+4,0+\y)--(\x+3,0+\y)--(\x+3,1+\y)--(\x+4,1+\y)--(\x+4,0+\y)--(\x+3,1+\y);
\draw (\x+3,0+\y)--(\x+4,1+\y);
\node at (\x+3.5,0.2+\y) {$b$}; 
\node at (\x+3.5,0.8+\y) {$l$}; 
\node at (\x+4.2,0.5+\y) {$1-i$}; 
\node at (\x+3.2,0.5+\y) {$i$}; 
}


\foreach \x in {-2}
\foreach \y in {-3}
{
\draw (\x+4,0+\y)--(\x+3,0+\y)--(\x+3,1+\y)--(\x+4,1+\y)--(\x+4,0+\y)--(\x+3,1+\y);
\draw (\x+3,0+\y)--(\x+4,1+\y);
\node at (\x+3.5,0.2+\y) {$l$}; 
\node at (\x+3.5,0.8+\y) {$b$}; 
\node at (\x+4.2,0.5+\y) {$1-i$}; 
\node at (\x+3.2,0.5+\y) {$i$}; 
}


\foreach \x in {4}
\foreach \y in {-3}
{
\draw (\x+0,0+\y)--(\x+2,0+\y)--(\x+2,1+\y)--(\x+0,1+\y)--(\x+0,0+\y);
\draw (\x+1,1+\y)--(\x+1,0+\y);

\draw (\x+0,0+\y)--(\x+0.5,0.5+\y)--(\x+0,1+\y);\draw (\x+2,0+\y)--(\x+1.5,0.5+\y)--(\x+2,1+\y);
\draw (\x+0.5, 0.5+\y)--(\x+1.5,0.5+\y);
\node at (\x+0.2,0.5+\y) {$i$}; 
\node at (\x+1.8,0.5+\y) {$i$};
\node at (\x+0.5,0.8+\y) {$b$}; 
\node at (\x+0.5,0.2+\y) {$r$};
\node at (\x+1.5,0.8+\y) {$r$}; 
\node at (\x+1.5,0.2+\y) {$b$};
}


\foreach \x in {8}
\foreach \y in {-3}
{
\draw (\x+0,0+\y)--(\x+1,0+\y)--(\x+1,2+\y)--(\x+0,2+\y)--(\x+0,0+\y);
\draw (\x+0,1+\y)--(\x+1,1+\y);

\draw (\x+0,0+\y)--(\x+0.5,0.5+\y)--(\x+1,0+\y);\draw (\x+0,2+\y)--(\x+0.5,1.5+\y)--(\x+1,2+\y);
\draw (\x+0.5, 0.5+\y)--(\x+0.5,1.5+\y);
\node at (\x+0.5,0.2+\y) {$b$}; \node at (\x+0.5,1.8+\y) {$r$};
\node at (\x+0.2,0.5+\y) {$i$}; 
\node at (\x+1.2,0.5+\y) {$1-i$};
\node at (\x+0.2,1.5+\y) {$i$}; 
\node at (\x+1.2,1.5+\y) {$1-i$};
}


\foreach \x in {7}
\foreach \y in {-3}
{
\draw (\x+4,0+\y)--(\x+6,0+\y)--(\x+6,2+\y)--(\x+4,2+\y)--(\x+4,0+\y)--(\x+4.5,0.5+\y)--(\x+5.5,0.5+\y)--(\x+5.5, 1.5+\y)--(\x+4.5, 1.5+\y)--(\x+4.5, 0.5+\y);
\draw (\x+4.5, 1.5+\y)--(\x+4,2+\y); \draw (\x+5.5, 1.5+\y)--(\x+6,2+\y);
\draw (\x+5.5, 0.5+\y)--(\x+6,0+\y); \draw (\x+4,1+\y)--(\x+6,1+\y); \draw (\x+5,0+\y)--(\x+5,2+\y);

\node at (\x+4.2,1.5+\y) {$i$}; 
\node at (\x+5.8,0.5+\y) {$i$}; 
\node at (\x+3.8,0.5+\y) {$1-i$}; 
\node at (\x+6.2,1.5+\y) {$1-i$}; 

\node at (\x+4.5,1.8+\y) {$r$}; 
\node at (\x+5.5,1.8+\y) {$b$}; 
\node at (\x+5.5,0.2+\y) {$b$}; 
\node at (\x+4.5,0.2+\y) {$b$}; 
}

\node at (1.5,0.2) {$V_0$}; 
\node at (4.5,0.2) {$V_{1x}$};
\node at (4.5,2.2) {$V_{1y}$}; 
\node at (7.5,0.2) {$C_0$}; 
\node at (11.5,0.2) {$C_1$}; 

\node at (-0.5,-3) {$F_i$}; 
\node at (2.5,-3) {$L_i$}; 
\node at (6.5,-3) {$R_i$}; 
\node at (9.5,-3) {$X_{ii}$}; 
\node at (13.8,-3) {$X_{i,1-i}$}; 

\end{tikzpicture}
\end{center}
\caption{The variables ($V_0, V_{1x}, V_{1y}$), clauses ($C_0, C_1$), forwarders ($F_i, i=0,1$), left anchors ($L_i, i=0,1$), right anchors ($R_i, i=0,1$),  crossovers ($X_{ii}$ and $X_{i,1-i}, i=0,1$) (modified).}\label{fig_wang_2}  
\end{figure}

Similarly, the colors of the horizontal sides of the Wang tiles of each column of the tiling also change alternatively, between the color $b$ and color $l$. So the top and bottom sides of the forwarders $F_i (i=1,2)$ are $l$ and $b$, respectively. The top and bottom sides of the left anchors $L_i (i=1,2)$ are $b$ and $l$, respectively. After these modifications, there is really no difference in the functionality of these two kinds of tiles. The forwarders and left anchors are placed side by side to forward the signals from the left to the right. Furthermore, unlike the proof of Theorem \ref{thm_main} where only the left anchors are used in controlling the positions of the crossover tiles from the bottom boundary of the region, the forwarders and left anchors now work together to control the crossover tiles as we will explain later.

The variable $V_1$ for setting the true value $1$ in Theorem \ref{thm_main} are replaced by two variables $V_{1x}$ and $V_{1y}$ to avoid having the same colors in their horizontal sides. The two new tiles $V_{1x}$ and $V_{1y}$ are obtained from $V_1$ by changing the color of the north side or south side from $b$ to $l$, respectively. The color of the north side of $V_0$ is also changed from $b$ to $l$. The colors of the boundary on the left part of the region are changed slightly according to this change (see Figure \ref{fig_region_2}).

\begin{figure}[H]
\begin{center}
\begin{tikzpicture}[scale=0.6]

\draw [fill=orange!20,draw=none] (-1,0)--(0,0)--(0,3)--(-1,3)--(-1,0);
\draw [fill=orange!20,draw=none] (-1,4)--(0,4)--(0,7)--(-1,7)--(-1,4);
\draw [fill=orange!20,draw=none] (-1,8)--(0,8)--(0,11)--(-1,11)--(-1,8);

\draw [fill=lime!20,draw=none]  (0,0)--(8,0)--(8,11)--(0,11)--(0,0);
\draw [fill=lime!40,draw=none]  (8,0)--(12,0)--(11,5)--(12,9)--(11,11)--(8,11)--(8,0);
\draw [fill=lime!40,draw=none]  (12.5,0)--(11.5,5)--(12.5,9)--(11.5,11)--(15,11)--(15,0)--(12.5,0);
\draw [fill=lime!20,draw=none]  (15,0)--(15,11)--(20,11)--(20,0)--(15,0);

\foreach \y in {0,4,8}
{
\draw [fill=violet!15,draw=none] (20,0+\y)--(22,0+\y)--(22,2+\y)--(21,2+\y)--(21,3+\y)--(20,3+\y)--(20,0+\y);
\draw [fill=violet!30,draw=none] (22,0+\y)--(24,0+\y)--(24,1+\y)--(23,1+\y)--(23,2+\y)--(22,2+\y)--(22,0+\y);
}

\draw (-1,11)--(-1,8)--(0,8)--(0,7)--(-1,7)--(-1,4)--(0,4)--(0,3)--(-1,3)--(-1,0)--(12,0)--(11,5)--(12,9)--(11,11)--(-1,11);
\draw (12.5,0)--(11.5,5)--(12.5,9)--(11.5,11)--(21,11)--(21,10)--(23,10)--(23,9)--(24,9)--(24,8)--(20,8)--
(20,7)--(21,7)--(21,6)--(23,6)--(23,5)--(24,5)--(24,4)--(20,4)--
(20,3)--(21,3)--(21,2)--(23,2)--(23,1)--(24,1)--(24,0)--(22,0)--(12.5,0);

\foreach \y in {1,2,3,5,6,7,9,10,11}
{
\node at (-1.3,\y-0.5) {$v$}; 
}

\foreach \x in {0}
\foreach \y in { 3.2,  7.2,  11.2}
{
\node at (\x-0.5,\y) {\small $l$}; 
}

\foreach \x in {0}
\foreach \y in {-0.2,   3.8,  7.8}
{
\node at (\x-0.5,\y) {\small $b$}; 
}

\foreach \x in {21}
\foreach \y in {3.2, 7.2,  11.2}
{
\node at (\x-0.5,\y) {\small $b$}; 
}

\foreach \x in {21}
\foreach \y in {-0.2, 3.8,  7.8 }
{
\node at (\x-0.5,\y) {\small $l$}; 
}

\foreach \x in {22}
\foreach \y in {-0.2, 2.2, 3.8, 6.2, 7.8, 10.2}
{
\node at (\x-0.5,\y) {\small $l$}; 
}

\foreach \x in {23}
\foreach \y in {-0.2, 2.2, 3.8, 6.2, 7.8, 10.2}
{
\node at (\x-0.5,\y) {\small $b$}; 
}

\foreach \x in {24}
\foreach \y in {-0.2,   3.8, 7.8 }
{
\node at (\x-0.5,\y) {\small $b$}; 
}

\foreach \x in {24}
\foreach \y in { 1.2,   5.2,  9.2}
{
\node at (\x-0.5,\y) {\small $l$}; 
}

\foreach \y in {3.5, 7.5}
{
\node at (-0.2,\y) {\small $0$}; 
}

\foreach \y in {3,7,11}
{
\node at (21.3,\y-0.5) {\small $0'$}; 
\node at (23.3,\y-1.5) {\small $0'$}; 
\node at (24.3,\y-2.5) {\small $1$}; 
}

\draw [<->,thick] (0,-1)--(20,-1);  
\node at (10,-0.6) {even width};
\draw (0,-1.5)--(0,-0.5); 
\draw (20,-1.5)--(20,-0.5); 

\end{tikzpicture}
\end{center}
\caption{The simply connected region (modified).}\label{fig_region_2}
\end{figure}


\begin{figure}[H]
\begin{center}
\begin{tikzpicture}[scale=0.6]

\draw (0,0)--(4,0)--(4,1)--(3,1)--(3,2)--(2,2)--(1,2)--(1,3)--(0,3)--(0,0); 

\draw (0,2)--(1,2)--(1,0); \draw (0,1)--(3,1);

\draw (2,0)--(2,2); \draw (3,0)--(3,1);

\node at (1.3,2.5) {\small $0'$}; 
\node at (3.3,1.5) {\small $0'$}; 
\node at (4.3, 0.5) {\small $1$}; 

\node at (0.5, 2.5) {\small $C_0$}; \node at (0.5, 1.5) {\small $F_0$};  \node at (1.5, 1.5) {\small $F_1$};  \node at (2.5, 1.5) {\small $C_0$}; 
\node at (0.5, 0.5) {\small $L_1$}; \node at (1.5, 0.5) {\small $L_0$}; \node at (2.5, 0.5) {\small $F_1$}; \node at (3.5, 0.5) {\small $F_0$};

\foreach \x in {8}
{
\draw (\x+0,0)--(\x+4,0)--(\x+4,1)--(\x+3,1)--(\x+3,2)--(\x+2,2)--(\x+1,2)--(\x+1,3)--(\x+0,3)--(\x+0,0); 

\draw (\x+0,2)--(\x+1,2)--(\x+1,0); \draw (\x+0,1)--(\x+2,1);

\draw (\x+2,0)--(\x+2,2);

\node at (\x+1.3,2.5) {\small $0'$}; 
\node at (\x+3.3,1.5) {\small $0'$}; 
\node at (\x+4.3, 0.5) {\small $1$}; 

\node at (\x+0.5, 2.5) {\small $C_0$}; 
\node at (\x+0.5, 1.5) {\small $F_1$};  \node at (\x+1.5, 1.5) {\small $F_0$};   
\node at (\x+0.5, 0.5) {\small $L_0$}; \node at (\x+1.5, 0.5) {\small $L_1$}; \node at (\x+2.7, 0.7) {\small $C_1$};  

}

\foreach \x in {16}
{
\draw (\x+0,0)--(\x+4,0)--(\x+4,1)--(\x+3,1)--(\x+3,2)--(\x+2,2)--(\x+1,2)--(\x+1,3)--(\x+0,3)--(\x+0,0); 

\draw (\x+0,1)--(\x+2,1);

\draw (\x+2,0)--(\x+2,2);  
\draw (\x+1,0)--(\x+1,1);

\node at (\x+1.3,2.5) {\small $0'$}; 
\node at (\x+3.3,1.5) {\small $0'$}; 
\node at (\x+4.3, 0.5) {\small $1$};

\node at (\x+0.7, 1.7) {\small $C_1$};   
\node at (\x+0.5, 0.5) {\small $L_0$}; \node at (\x+1.5, 0.5) {\small $L_1$}; \node at (\x+2.7, 0.7) {\small $C_1$};  

}


\foreach \y in {0,1,2}
{
\draw [->, thick] (-1,0.5+\y)--(0,0.5+\y);
}
\node at (-1.2, 0.5) {\small $1$};  
\node at (-1.2, 1.5) {\small $0$};  
\node at (-1.2, 2.5) {\small $0$};

\foreach \x in {8}
{
\foreach \y in {0,1,2}
{
\draw [->, thick] (\x+-1,0.5+\y)--(\x+0,0.5+\y);
}
\node at (\x+-1.2, 0.5) {\small $0$};  
\node at (\x+-1.2, 1.5) {\small $1$};  
\node at (\x+-1.2, 2.5) {\small $0$};  
}

\foreach \x in {16}
{
\foreach \y in {0,1,2}
{
\draw [->, thick] (\x+-1,0.5+\y)--(\x+0,0.5+\y);
}
\node at (\x+-1.2, 0.5) {\small $0$};  
\node at (\x+-1.2, 1.5) {\small $0$};  
\node at (\x+-1.2, 2.5) {\small $1$};  
}

\end{tikzpicture}
\end{center}
\caption{The clause sub-regions (modified).}\label{fig_clause_region_2}
\end{figure}

The clause tiles are modified too. The clause tile $C_0$ is changed so that the two horizontal sides have different colors. The clause tile $C_1$ is changed by adding one more Wang tile because the signals are to be checked every two columns (recall that the signals are switching back and forth between $0$ and $1$ along the transmission path). The new tile $C_1$ redirects the signal $1$ one row downwards and two columns to the right along the transmission path. The colors of the boundary of the clause area are modified accordingly too. The partial tilings corresponding to the $3$ acceptable signal sequences for each clause, namely $001$, $010$, and $100$ are illustrated in Figure \ref{fig_clause_region_2}.

The right anchors $R_i (i=1,2)$ remain the same. But the crossovers $X_{ij} (i,j=1,2)$ need to be changed. Two of them, $X_{ii}$ for swapping two adjacent identical signals (both $1$ or both $0$), are modified in a very subtle way. The colors of the south side of $X_{ii} (i=1,2)$ are changed from $l$ to $b$. The other two, $X_{i,1-i}$ for swapping two different signals, are enlarged to a $2\times 2$ generalized Wang tiles to avoid having the same color on their vertical sides.


\begin{figure}[H]
\begin{center}
\begin{tikzpicture}[scale=0.6]

\draw [fill=lime!15] (0,0)--(8,0)--(8,11)--(0,11)--(0,0);

\foreach \y in {5,7,9}
{
\draw [fill=lime!45,draw=none]  (0,0+\y)--(0,1+\y)--(1,1+\y)--(1,0+\y)--(0,0+\y);
}
\foreach \y in {0,2,6,8,10}
\foreach \x in {1}
{
\draw [fill=lime!45,draw=none]  (\x+0,0+\y)--(\x+0,1+\y)--(\x+1,1+\y)--(\x+1,0+\y)--(\x+0,0+\y);
}
\foreach \y in {1,3,7,9}
\foreach \x in {2}
{
\draw [fill=lime!45,draw=none]  (\x+0,0+\y)--(\x+0,1+\y)--(\x+1,1+\y)--(\x+1,0+\y)--(\x+0,0+\y);
}
\foreach \y in {0,2,4,8,10}
\foreach \x in {3}
{
\draw [fill=lime!45,draw=none]  (\x+0,0+\y)--(\x+0,1+\y)--(\x+1,1+\y)--(\x+1,0+\y)--(\x+0,0+\y);
}
\foreach \y in {1,3,5,9}
\foreach \x in {4}
{
\draw [fill=lime!45,draw=none]  (\x+0,0+\y)--(\x+0,1+\y)--(\x+1,1+\y)--(\x+1,0+\y)--(\x+0,0+\y);
}
\foreach \y in {0,2,4,6,10}
\foreach \x in {5}
{
\draw [fill=lime!45,draw=none]  (\x+0,0+\y)--(\x+0,1+\y)--(\x+1,1+\y)--(\x+1,0+\y)--(\x+0,0+\y);
}
\foreach \y in {1,3,5,7}
\foreach \x in {6}
{
\draw [fill=lime!45,draw=none]  (\x+0,0+\y)--(\x+0,1+\y)--(\x+1,1+\y)--(\x+1,0+\y)--(\x+0,0+\y);
}
\foreach \y in {0,2,4,6,8}
\foreach \x in {7}
{
\draw [fill=lime!45,draw=none]  (\x+0,0+\y)--(\x+0,1+\y)--(\x+1,1+\y)--(\x+1,0+\y)--(\x+0,0+\y);
}


\draw   (0,0)--(8,0)--(8,11)--(0,11)--(0,0);

\foreach \y in {1,2,5,6,7,8,9,10,11}
{
\draw [->, thick] (-0.5,\y-0.5)--(8.5, \y-0.5);
}

\foreach \x in {0,...,6}
{
\draw [fill=white] (-\x+6,11-\x)--(-\x+8,11-\x)--(-\x+8,10-\x)--(-\x+6,10-\x)--(-\x+6,11-\x);
\node at (7-\x, 10.5-\x) {\small $R$}; 
}

\draw [fill=lime!45] (0,0)--(0,1)--(1,1)--(1,0)--(0,0);
\node at (0.5, 0.5) {\small $F$}; 

\draw [fill=lime!15] (0,1)--(0,2)--(1,2)--(1,1)--(0,1);
\node at (0.5, 1.5) {\small $L$}; 

\draw [fill=white] (1,2)--(1,4)--(0,4)--(0,2)--(1,2);
\node at (0.5,3.7) {\small $X_{ii}$}; 
\draw [->, thick, red] (-0.5, 3.5)--(0,3.5)--(1, 2.5)--(8.5, 2.5);
\draw [->, thick, blue] (-0.5, 2.5)--(0,2.5)--(1, 3.5)--(8.5, 3.5);

\foreach \x in {1,3,5,7}
{
\node at (\x-0.5,11.3) {\small $b$}; 
}
\foreach \x in {2,4,6}
{
\node at (\x-0.5,11.3) {\small $l$}; 
}

\foreach \x in {2,4,6}
{
\node at (\x+0.5,-0.3) {\small $l$}; 
}
\foreach \x in {3,5,7}
{
\node at (\x+0.5,-0.3) {\small $b$}; 
}

\node at (7.5,11.3) {\small $r$}; 
\node at (0.5,-0.3) {\small $b$}; \node at (1.5,-0.3) {\small $b$}; 


\foreach \z in {13}
{
\draw [fill=lime!15] (\z+0,0)--(\z+8,0)--(\z+8,11)--(\z+0,11)--(\z+0,0);

\foreach \y in {5,7,9}
{
\draw [fill=lime!45,draw=none]  (\z+0,0+\y)--(\z+0,1+\y)--(\z+1,1+\y)--(\z+1,0+\y)--(\z+0,0+\y);
}
\foreach \y in {0,2,6,8,10}
\foreach \x in {1}
{
\draw [fill=lime!45,draw=none]  (\z+\x+0,0+\y)--(\z+\x+0,1+\y)--(\z+\x+1,1+\y)--(\z+\x+1,0+\y)--(\z+\x+0,0+\y);
}
\foreach \y in {1,3,7,9}
\foreach \x in {2}
{
\draw [fill=lime!45,draw=none]  (\z+\x+0,0+\y)--(\z+\x+0,1+\y)--(\z+\x+1,1+\y)--(\z+\x+1,0+\y)--(\z+\x+0,0+\y);
}
\foreach \y in {0,2,4,8,10}
\foreach \x in {3}
{
\draw [fill=lime!45,draw=none]  (\z+\x+0,0+\y)--(\z+\x+0,1+\y)--(\z+\x+1,1+\y)--(\z+\x+1,0+\y)--(\z+\x+0,0+\y);
}
\foreach \y in {1,3,5,9}
\foreach \x in {4}
{
\draw [fill=lime!45,draw=none]  (\z+\x+0,0+\y)--(\z+\x+0,1+\y)--(\z+\x+1,1+\y)--(\z+\x+1,0+\y)--(\z+\x+0,0+\y);
}
\foreach \y in {0,2,4,6,10}
\foreach \x in {5}
{
\draw [fill=lime!45,draw=none]  (\z+\x+0,0+\y)--(\z+\x+0,1+\y)--(\z+\x+1,1+\y)--(\z+\x+1,0+\y)--(\z+\x+0,0+\y);
}
\foreach \y in {1,3,5,7}
\foreach \x in {6}
{
\draw [fill=lime!45,draw=none]  (\z+\x+0,0+\y)--(\z+\x+0,1+\y)--(\z+\x+1,1+\y)--(\z+\x+1,0+\y)--(\z+\x+0,0+\y);
}
\foreach \y in {0,2,4,6,8}
\foreach \x in {7}
{
\draw [fill=lime!45,draw=none]  (\z+\x+0,0+\y)--(\z+\x+0,1+\y)--(\z+\x+1,1+\y)--(\z+\x+1,0+\y)--(\z+\x+0,0+\y);
}


\draw   (\z+0,0)--(\z+8,0)--(\z+8,11)--(\z+0,11)--(\z+0,0);

\foreach \y in {1,2,5,6,7,8,9,10,11}
{
\draw [->, thick] (\z+-0.5,\y-0.5)--(\z+8.5, \y-0.5);
}

\foreach \x in {0,...,6}
{
\draw [fill=white] (\z+-\x+6,11-\x)--(\z+-\x+8,11-\x)--(\z+-\x+8,10-\x)--(\z+-\x+6,10-\x)--(\z+-\x+6,11-\x);
\node at (\z+7-\x, 10.5-\x) {\small $R$}; 
}

\draw [fill=lime!45] (\z+0,0)--(\z+0,1)--(\z+1,1)--(\z+1,0)--(\z+0,0);
\node at (\z+0.5, 0.5) {\small $F$}; 

\draw [fill=lime!15] (\z+0,1)--(\z+0,2)--(\z+1,2)--(\z+1,1)--(\z+0,1);
\node at (\z+0.5, 1.5) {\small $L$}; 

\draw [fill=white] (\z+2,2)--(\z+2,4)--(\z+0,4)--(\z+0,2)--(\z+2,2);
\node at (\z+1,3.6) {\small $X_{i,1-i}$}; 
\draw [->, thick, red] (\z+-0.5, 3.5)--(\z+0,3.5)--(\z+2, 2.5)--(\z+8.5, 2.5);
\draw [->, thick, blue] (\z+-0.5, 2.5)--(\z+0,2.5)--(\z+2, 3.5)--(\z+8.5, 3.5);

\foreach \x in {1,3,5,7}
{
\node at (\z+\x-0.5,11.3) {\small $b$}; 
}
\foreach \x in {2,4,6}
{
\node at (\z+\x-0.5,11.3) {\small $l$}; 
}

\foreach \x in {2,4,6}
{
\node at (\z+\x+0.5,-0.3) {\small $l$}; 
}
\foreach \x in {3,5,7}
{
\node at (\z+\x+0.5,-0.3) {\small $b$}; 
}

\node at (\z+7.5,11.3) {\small $r$}; 
\node at (\z+0.5,-0.3) {\small $b$}; \node at (\z+1.5,-0.3) {\small $b$}; 

}

\end{tikzpicture}
\end{center}
\caption{The crossover sub-regions (modified).}\label{fig_crossover_region_2}
\end{figure}

With these modifications to the crossovers, new crossover sub-regions of width $8$ in the central part are illustrated in Figure \ref{fig_crossover_region_2}. This region will force a swap for the signals on the $8$-th and $9$-th rows. The colors on the top boundary of this sub-region are $(bl)^3br$. In general, for a sub-region of width $2k$, the colors on the top boundary are $(bl)^{k-1}br$; and for a sub-region of width $2k+1$, the colors on the top boundary are $l(bl)^{k-1}br$. The unique color $r$ on the top boundary will force a tile $R$ to be placed on the upper right corner of this sub-region, and then an array of tiles $R$ will extend along the southwest direction. The rest area is filled with the tiles $F$ and $L$ in a checkerboard pattern. In Figure \ref{fig_crossover_region_2}, the darker green and lighter green represent forwarders $F$ and left anchors $L$, respectively. We can see that the right anchors are compatible with the checkerboard pattern, as the checkerboard patterns on the two sides of the array of right anchors match. Therefore, the checkerboard pattern is determined by the colors of the top boundary and extends downwards.

The bottom boundary of the crossover sub-region is colored according to the principle that, except the color of the first unit segment from the left, all other colors are compatible with the checkerboard pattern which is determined by the colors of the top boundary. Therefore, the colors on the bottom depend on the parity of the height of the sub-region. In the example of Figure \ref{fig_crossover_region_2}, where the height is odd (i.e. $11$), the colors of the bottom boundary are $bb(lb)^3$. If the height were even, the colors of the bottom boundary would have been $ll(bl)^3$. In other words, except for the first two colors, the color sequence of the bottom boundary is alternating between $b$ and $l$. Note that the only tiles that are not compatible with the checkerboard pattern are the crossover tiles $X_{ii} (i=1,2)$ and the first column of the crossover tiles $X_{i,1-i} (i=1,2)$. The incompatibility of the first color at the bottom boundary will force a crossover tile to be placed under the tile $R$ on the first column of this sub-region. The signals coming from the left determine whether $X_{ii}$ (left of Figure \ref{fig_crossover_region_2}) or $X_{i,1-i}$ (right of Figure \ref{fig_crossover_region_2}) to be used.

With all the above modifications, the total number of Wang tiles becomes $29$, this completes the proof. \end{proof}

Lemma \ref{lem_wang} improves a result of \cite{y_thesis} under the same condition. In \cite{y_thesis}, it is proved that there exist a set $T$ of $33$ Wang tiles with different colors on parallel sides such that the tiling problem for simply connected regions with $T$ is \textsf{NP}-complete.

We need the following lemma from \cite{y_thesis} to reduce a set of Wang tiles (with different colors on parallel sides) to a set of rectangles with respect to the tiling problem for simply connected regions.

\begin{Lemma}[\cite{y_thesis}]\label{lem_jed_yang}
Let $T$ be a set of Wang tiles such that the colors of parallel sides of each tile are different. If the problem of tiling simply connected regions with $T$ is \textsf{NP}-complete, then a set $R$ of rectangles can be constructed based on $T$ such that the problem of tiling simply connected regions with $R$ is also \textsf{NP}-complete.
\end{Lemma}  

More precisely, the set of rectangles in the proof of Lemma \ref{lem_jed_yang} is constructed in the following way \cite{y_thesis}. Let $n$ be the total number of different colors in the set $T$ of Wang tiles, and let $M=100\cdot 5^n$. Relabel the $n$ different colors of Wang tiles by integers from $1$ to $n$. For each tile $\tau$ in $T$, let $N_{\tau}$, $S_{\tau}$, $W_{\tau}$ and $E_{\tau}$ be the colors (i.e. integers) of the north, south, west and east side of the tile $\tau$. Then the set of $R$ consists of the following rectangles \cite{y_thesis}.

\begin{itemize}
    \item A rectangle $f$ of size $11M \times 34 M$.
    \item A rectangle $h$ of size $31M \times 11 M$.
    \item A rectangle $v$ of size $34M \times 14M$.
    \item For each tile $\tau \in T$, construct $5$ rectangles: a rectangle $w$ of size $\big ( 14M+5^{E_{\tau}}-5^{W_{\tau}} \big ) \times \big ( 31M+5^{S_{\tau}}-5^{N_{\tau}} \big )$, a rectangle $s_1$ of size $\big ( 10M+5^{W_{\tau}} \big ) \times \big ( 10M-5^{N_{\tau}} \big )$, a rectangle $s_2$ of size $\big ( 10M-5^{E_{\tau}} \big ) \times \big ( 10M-5^{N_{\tau}} \big )$, a rectangle $s_3$ of size $\big ( 10M+5^{W_{\tau}} \big ) \times \big ( 10M+5^{S_{\tau}} \big )$ and a rectangle $s_4$ of size $\big ( 10M-5^{E_{\tau}} \big ) \times \big ( 10M+5^{S_{\tau}} \big )$.
\end{itemize}

By the above reduction method, if two Wang tiles have the same color on their west sides and north sides, respectively, they will generate the rectangles $s_1$ of the same size. Similarly, if two Wang tiles have the same color in two other adjacent sides (i.e., north and east, south and west, or south and east), they will generate the same rectangles $s_2$, $s_3$, or $s_4$. So we count the number of different rectangles of the types $w$ and $s_i (i=1,2,3,4)$ in Table \ref{tbl_rect}.

\begin{table}[H]
\begin{center}
\begin{tabular}{|c|r|r|r|r|r|r|}
\hline
Rectangles          &  Left anchors   &  Forwarders   &  Right anchors   &  Variables   &   Crossovers    &  Clauses    \\ \hline
$w$             & 2     &  2   &  4   &  5   &   12    &    4   \\ \hline
$s_1$           & 2      &   2  &   2  &  4   &    10   &    2     \\ \hline
$s_2$             & 2     &  2   &  4   &  2   &    6   &    2      \\ \hline
$s_3$             & 2     &  2   &  4   &   4  &   8    &     2     \\ \hline
$s_4$             & 2     &   2  &   2  &  2   &    6   &    3      \\ \hline
Total            & 10     &   10 &   16  &  17  &    42   &    13      \\ \hline
\end{tabular}
\end{center}
\caption{Counting rectangles}\label{tbl_rect}
\end{table}

We count the number of rectangles on Table \ref{tbl_rect} column by column from left to right. For the first two columns, we have two left anchors $L_i (i=1,2)$ and two forwarders $F_i (i=1,2)$, which generate a total of $4\times 5 =20$ rectangles. By cutting apart the right anchors $R_i (i=1,2)$, we get $4$ Wang tiles. Two of them have the same color on the north and west sides with $L_1$ and $L_2$, so only $2$ new rectangles of the type $s_1$ are introduced by the right anchors. Similarly, two of them have the same color on the south and east sides with $F_1$ and $F_2$, so the new rectangles of type $s_4$ are also only $2$. The rest three columns are counted in the same way, by ignoring the rectangles which have appeared in the preceding columns.

A special technique is applied to crossovers which further decreases the numbers of rectangles of the types $s_2$ and $s_4$. When breaking down a generalized Wang tile into (single square) Wang tiles, we always assigned new colors to the sides of each cut. In this way, those Wang tiles must be put together to form the original generalized Wang tile in any tilings. Cutting apart the crossover $X_{i,1-i}$ introduces four pairs of new sides. Three pairs of them are enough to ensure the Wang tiles to be put back together to the original $2\times 2$ crossovers. So we color the fourth pair with an existing color $l$ as illustrated in Figure \ref{fig_cut}. Using the existing color $l$ for the cut sides enables us to further decrease the number of rectangles of type $s_2$ and $s_4$.


\begin{figure}[H]
\begin{center}
\begin{tikzpicture}

\foreach \x in {-4}
\foreach \y in {0}
{
\draw (\x+4,0+\y)--(\x+6,0+\y)--(\x+6,2+\y)--(\x+4,2+\y)--(\x+4,0+\y)--(\x+4.5,0.5+\y)--(\x+5.5,0.5+\y)--(\x+5.5, 1.5+\y)--(\x+4.5, 1.5+\y)--(\x+4.5, 0.5+\y);
\draw (\x+4.5, 1.5+\y)--(\x+4,2+\y); \draw (\x+5.5, 1.5+\y)--(\x+6,2+\y);
\draw (\x+5.5, 0.5+\y)--(\x+6,0+\y); \draw (\x+4,1+\y)--(\x+6,1+\y); \draw (\x+5,0+\y)--(\x+5,2+\y);

\node at (\x+4.2,1.5+\y) {$i$}; 
\node at (\x+5.8,0.5+\y) {$i$}; 
\node at (\x+3.8,0.5+\y) {$1-i$}; 
\node at (\x+6.2,1.5+\y) {$1-i$}; 

\node at (\x+4.5,1.8+\y) {$r$}; 
\node at (\x+5.5,1.8+\y) {$b$}; 
\node at (\x+5.5,0.2+\y) {$b$}; 
\node at (\x+4.5,0.2+\y) {$b$};

\node [red] at (\x+5.6,1.2+\y) {$l$}; 
}
\draw [red, very thick] (1,1)--(2,1);

\end{tikzpicture}
\end{center}
\caption{Cutting apart $X_{i,1-i}, (i=0,1$).}\label{fig_cut}  
\end{figure}

Finally, by adding up the $w$ and $s_i (i=1,2,3,4)$ rectangles of the six columns of Table \ref{tbl_rect}, as well as $3$ rectangles $f$, $h$ and $v$, the total number of rectangles is
$$10+10+16+17+42+13+3=111.$$

So by combining Lemma \ref{lem_wang} and Lemma \ref{lem_jed_yang}, we have proved the following theorem.

\begin{Theorem}
    There exists a set $R$ of $111$ rectangles such that the tiling simply connected region with $R$ is \textsf{NP}-complete.
\end{Theorem}

\section{Conclusions}\label{sec_con}

We prove the \textsf{NP}-completeness of tiling simply connected regions with a set of $23$ Wang tiles. As a consequence of a variant of this result, we also prove the \textsf{NP}-completeness of tiling simply connected regions with a set of $111$ rectangles. For tiling simply connected regions with a fixed set of rectangles, Pak and Yang made the following conjecture, which deserves further investigation.

\begin{Conjecture}[\cite{py13}]
    There exists a set $R$ of $3$ rectangles such that tiling simply connected regions with $R$ is \textsf{NP}-complete.
\end{Conjecture}

For tiling finite regions with a fixed set of Wang tiles, can we further improve Theorem \ref{thm_main} by using fewer number of Wang tiles?

\begin{Problem}
    Is there a set $W$ of less than 23 Wang tiles such that tiling simply connected regions with $W$ is \textsf{NP}-complete?
\end{Problem}

For general regions, as we have mentioned in the first section of this paper, there exists a set $W$ of $5$ Wang tiles such that tiling general finite regions with $W$ is \textsf{NP}-complete. For a single Wang tile, it is trivial to determine whether it can tile a finite region in polynomial time. For a set of $2$ or $3$ Wang tiles, the problem is also solvable in polynomial time by the following Theorem \ref{thm_p}. To prove Theorem \ref{thm_p}, we make use of a parameter for sets of Wang tiles which was first introduced by Jeandel and Rolin \cite{jr12}, and was later named by Yang and Zhang as \textit{color deficiency} \cite{yz24b}.

\begin{Definition}[Color deficiency]
    Given a set $T$ of Wang tiles, let $n(T)=|T|$ be the number of tiles in the set, and let $c(T)$ be the maximum number of different colors among the four sides in the set. The color deficiency, denoted by $cd(T)$ is defined by $cd(T)=n(T)-c(T).$
\end{Definition}

\begin{Theorem}\label{thm_p}
The problem of tiling general finite regions with a set of $2$ or $3$ Wang tiles is solvable in polynomial time.
\end{Theorem}

\begin{proof}
    Let $W$ be a set of Wang tiles. If $|W|=2$, consider two cases according to the color deficiency of $W$. Let $W=\{X,Y\}$, namely the two Wang tiles are denoted by $X$ and $Y$.

\begin{itemize}
    \item If $cd(W)=1$, then for each of the four sides, the colors of $X$ and $Y$ are the same. Therefore, $X$ and $Y$ are identical, which contradicts the hypothesis that $|W|=2$.
    \item If $cd(W)=0$, then the color of at least one side of $X$ and $Y$ are different. Without loss of generality, assume the north side of $X$ and $Y$ are colored by $0$ and $1$, respectively. Let $R$ be an arbitrary finite region with each unit segment of its boundary labeled. Then Algorithm \ref{alg_top} solves the problem in polynomial time.
    
\begin{algorithm}[htb]
\caption{Tiling from one side (the top)}
\label{alg_top}
\begin{algorithmic}[1] 
\REQUIRE ~~\\ 
    The fixed set of Wang tiles $W=\{X,Y\}$, and a general region $R$.
\ENSURE ~~\\ 
    Output \textsf{yes} if $W$ tiles $R$, and \textsf{no} otherwise.
    \STATE If $R=\emptyset$, \textbf{return} \textsf{yes}. Otherwise, pick a unit segment $t$ on the top boundary of $R$, and let $s$ denote the unit square under $t$.
    \STATE If $t$ is labeled by color $0$, place the Wang tile $X$ on $s$. If $t$ is labeled by color $1$, place $Y$ on $s$. If $t$ is labeled by colors other than $0$ and $1$, \textbf{return} \textsf{no}.
    \STATE If the Wang tile placed on $s$ conflicts with the color of the boundary on other sides, \textbf{return} \textsf{no}; otherwise, let $R=R-\{s\}$ and \textbf{goto} step 1.
\end{algorithmic}
\end{algorithm}
\end{itemize}

Simply speaking, the two Wang tiles are \textit{distinguishable} from the north side, so we just tile the region $R$ square-by-square from the top boundary in Algorithm \ref{alg_top}. If the region can be tiled, the tiling is unique in this case.

If $|W|=3$, let $W=\{X,Y,Z\}$. We consider the following three cases according to the color deficiency of $W$.

\begin{itemize}
    \item If $cd(W)=2$, then $X$, $Y$ and $Z$ are in fact identical. This contradicts with the hypothesis that $|W|=3$.
    \item If $cd(W)=0$, then $X$, $Y$ and $Z$ are distinguishable from just one side (i.e. the colors of this side are different for $X$, $Y$ and $Z$). In this case, we can apply an algorithm similar to Algorithm \ref{alg_top}.
    \item If $cd(W)=1$, then each side of $X$, $Y$ and $Z$ has most two different colors, and at least one side has exactly two different colors. We further consider two subcases.
\end{itemize}

\begin{Fact}\label{fct_corner}
    Let $W=\{X,Y,Z\}$ be a set of $3$ Wang tiles, and $cd(W)=1$. If two neighboring sides both have two different colors, then the three Wang tiles $X$, $Y$ and $Z$ are distinguishable by some two neighboring sides. 
\end{Fact}

\begin{proof}[Proof of Fact \ref{fct_corner}]
    Without loss of generality, assume the set has two different colors on the north side and the west side. If the three Wang tiles can be distinguished by these two sides, we are done. Now suppose to the contrast that they are not distinguishable by these two sides, then two tiles are assigned the same colors on the north side and they are also assigned the same color on the west side. Without loss of generality, assume the tiles $X$ and $Y$ are colored $a$ on the west side and colored $0$ on the north side. By the hypothesis that the set has two colors in both the west side and the north side, the third tile $Z$ must be assigned colors different than that of $X$ and $Y$ on these two sides. Denote the colors on the west side and the north side of $Z$ by $b$ and $1$, respectively (see Figure \ref{fig_d_set}). Now $X$ and $Y$ must be assigned different colors in either the south side or the east side, otherwise, they are identical and contradict with $|W|=3$. Assume $X$ and $Y$ are colored differently on the east side, then $X$, $Y$ and $Z$ are distinguishable by two sides: north and east. This completes the proof of Fact \ref{fct_corner}.
\end{proof}

\begin{figure}[H]
\begin{center}
\begin{tikzpicture}[scale=0.5]

\foreach \x in {0}
{
\draw (\x+0,0)--(\x+2,0)--(\x+2,2)--(\x+0,2)--(\x+0,0);
\draw (\x+0,0)--(\x+2,2); \draw (\x+2,0)--(\x+0,2);
\node at (\x+1,1.6) {$0$};  
\node at (\x+0.4,1) {$a$};  \node at (\x+1, -0.8) {$X$}; 
}

\foreach \x in {4}
{
\draw (\x+0,0)--(\x+2,0)--(\x+2,2)--(\x+0,2)--(\x+0,0);
\draw (\x+0,0)--(\x+2,2); \draw (\x+2,0)--(\x+0,2);
\node at (\x+1,1.6) {$0$};
\node at (\x+0.4,1) {$a$};  \node at (\x+1, -0.8) {$Y$};
}

\foreach \x in {8}
{
\draw (\x+0,0)--(\x+2,0)--(\x+2,2)--(\x+0,2)--(\x+0,0);
\draw (\x+0,0)--(\x+2,2); \draw (\x+2,0)--(\x+0,2);
\node at (\x+1,1.6) {$1$};
\node at (\x+0.4,1) {$b$};  \node at (\x+1, -0.8) {$Z$};
}

\end{tikzpicture}
\end{center}
\caption{A distinguishable set of $3$ Wang tiles.}\label{fig_d_set}  
\end{figure}

If the condition in Fact \ref{fct_corner} is satisfied, then the set of $3$ Wang tiles are distinguishable by two neighboring sides. Without loss of generality, assume they are distinguishable by the north and east sides. More specifically, assume the colors of these two sides of $X$, $Y$ and $Z$ are $1a$, $1b$ and $2a$, respectively. Then Algorithm \ref{alg_corner} solves this case in polynomial time.

\begin{algorithm}[htb]
\caption{Tiling from the corner (the upper right corner)}
\label{alg_corner}
\begin{algorithmic}[1] 
\REQUIRE ~~\\ 
    The fixed set of Wang tiles $W=\{X,Y,Z\}$, and a general region $R$.
\ENSURE ~~\\ 
    Output \textsf{yes} if $W$ tiles $R$, and \textsf{no} otherwise.
    \STATE If $R=\emptyset$, \textbf{return} \textsf{yes}. Otherwise, pick a square $s$ at the upper right corner of $R$, let $t_1$ and $t_2$ denote the north side and east side of $s$, respectively (both $t_1$ and $t_2$ are on the boundary of $R$ because $u$ is a square at the upper right corner);
    \STATE If $t_1t_2$ is labeled by colors $1a$, $1b$ or $2a$, place the Wang tile $X$, $Y$ or $Z$ on $s$, respectively. If $t_1t_2$ is labeled by other color combinations, \textbf{return} \textsf{no}.
    \STATE If the Wang tile placed on $s$ conflicts with the color of the boundary on other sides, \textbf{return} \textsf{no}; otherwise, let $R=R-\{s\}$ and \textbf{goto} step 1.
\end{algorithmic}
\end{algorithm}

The other subcase for $cd(W)=1$ is that the condition in the Fact \ref{fct_corner} is not satisfied (i.e. no two neighboring sides of $X$, $Y$ and $Z$ both have two different colors). In this subcase, it is easy to see that there must exist two opposite sides of $X$, $Y$ and $Z$ such that both sides have exactly one color. Without loss of generality, we assume the north side of $X$, $Y$ and $Z$ are all colored $1$, and the south side of $X$, $Y$ and $Z$ are all colored $1$ too \footnote{It is possible that the north side and south side receive different colors, but this case is either trivial or can be treated similarly.}. Then the region $R$ can be divided into several horizontal bars $H$ of the form $m\times 1$, where each unit segment of the top and bottom of $H$ is labeled $1$, and left and right sides inherit the colors from the boundary of $R$. The set of Wang tiles can tile the region $R$ if and only if all horizontal bars can be tiled. Therefore, the problem is reduced to the tiling of horizontal bars. This problem is easy to be solved in polynomial time, and the set of $3$ Wang tiles in this subcase in fact has just a few possible combinations (up to renaming the colors). See Figure~\ref{fig_h_set} for an example of a possible set of $3$ different Wang tiles in this subcase. It is straightforward to see this set of Wang tiles can tile any horizontal bar $H$ of length at least $2$.

\begin{figure}[H]
\begin{center}
\begin{tikzpicture}[scale=0.5]

\foreach \x in {0}
{
\draw (\x+0,0)--(\x+2,0)--(\x+2,2)--(\x+0,2)--(\x+0,0);
\draw (\x+0,0)--(\x+2,2); \draw (\x+2,0)--(\x+0,2);
\node at (\x+1,0.4) {$1$}; \node at (\x+1,1.6) {$1$};
\node at (\x+0.4,1) {$a$}; \node at (\x+1.6,1) {$a$};
}

\foreach \x in {4}
{
\draw (\x+0,0)--(\x+2,0)--(\x+2,2)--(\x+0,2)--(\x+0,0);
\draw (\x+0,0)--(\x+2,2); \draw (\x+2,0)--(\x+0,2);
\node at (\x+1,0.4) {$1$}; \node at (\x+1,1.6) {$1$};
\node at (\x+0.4,1) {$a$}; \node at (\x+1.6,1) {$b$};
}

\foreach \x in {8}
{
\draw (\x+0,0)--(\x+2,0)--(\x+2,2)--(\x+0,2)--(\x+0,0);
\draw (\x+0,0)--(\x+2,2); \draw (\x+2,0)--(\x+0,2);
\node at (\x+1,0.4) {$1$}; \node at (\x+1,1.6) {$1$};
\node at (\x+0.4,1) {$b$}; \node at (\x+1.6,1) {$a$};
}

\end{tikzpicture}
\end{center}
\caption{A set of $3$ Wang tiles.}\label{fig_h_set}  
\end{figure}
    
By the above case-by-case discussion, the proof of Theorem \ref{thm_p} is complete.\end{proof}

Theorem \ref{thm_p} solves the cases of $|W|=2$ and $|W|=3$, so the only remaining case is the following problem.
    
\begin{Problem}
    What is the computational complexity of tiling general regions (i.e. allowing holes) with a set $W$ of $4$ Wang tiles? Is it solvable in polynomial time, or \textsf{NP}-complete?
\end{Problem}

\section*{Acknowledgements}
The first author was supported by the Research Fund of Guangdong University of Foreign Studies (Nos. 297-ZW200011 and 297-ZW230018), and the National Natural Science Foundation of China (No. 61976104).


\end{document}